\documentclass[12pt]{article}

\usepackage{amssymb}
\usepackage{amsmath}
\usepackage{amsthm}
\usepackage{upref}
\usepackage{enumerate}
\usepackage{mathtools}

\usepackage[utf8]{inputenc}
\usepackage[english]{babel}

\usepackage[a4paper,margin=2.5cm]{geometry}
\theoremstyle{plain}
\newtheorem{theorem}{Theorem}[section]

\newtheorem{lemma}[theorem]{Lemma}
\newtheorem{corollary}[theorem]{Corollary}
\newtheorem*{theorem2}{Theorem}

\theoremstyle{definition}

\newtheorem{assumption}[theorem]{Assumption}
\newtheorem{remark}[theorem]{Remark}

\numberwithin{equation}{section}

\newcommand{\abs}[1]{\lvert{#1}\rvert}

\newcommand{\norm}[1]{\lVert{#1}\rVert}
\newcommand{\normb}[1]{\bigl\lVert{#1}\bigr\rVert}

\newcommand{\ip}[2]{\langle{#1},{#2}\rangle}

\DeclareMathOperator{\ran}{Ran}

\DeclareMathOperator{\supp}{supp}

\newcommand{\vp}{\varphi}
\newcommand{\bR}{\mathbf{R}}
\newcommand{\bC}{\mathbf{C}}

\newcommand{\cH}{\mathcal{H}}
\newcommand{\cD}{\mathcal{D}}
\newcommand{\cO}{\mathcal{O}}

\newcommand{\cB}{\mathcal{B}}

\newcommand{\bZ}{\mathbf{Z}}

\newcommand{\bN}{\mathbf{N}}

\newcommand{\e}{{\rm e}}
\newcommand{\I}{{\rm i}}
\newcommand{\di}{\,{\rm d}}

\newcommand{\cHro}{\mathcal{H}^{\textup{ro}}}
\newcommand{\cHre}{\mathcal{H}^{\textup{re}}}
\newcommand{\cHroh}{\mathcal{H}^{\textup{ro}}_h}
\newcommand{\cHreh}{\mathcal{H}^{\textup{re}}_h}
\newcommand{\rfl}[1]{\breve{#1}}

\newcommand{\even}{\mathcal{E}}
\newcommand{\odd}{\mathcal{O}}
\newcommand{\R}{\mathcal{R}}
\newcommand{\DL}{H_0^{\textup{D}}}
\newcommand{\NL}{H_0^{\textup{N}}}
\newcommand{\DLh}{H_{0,h}^{\textup{D}}}
\newcommand{\NLh}{H_{0,h}^{\textup{N}}}
\newcommand{\DLp}{H^{\textup{D}}}
\newcommand{\NLp}{H^{\textup{N}}}
\newcommand{\DLhp}{H_h^{\textup{D}}}
\newcommand{\NLhp}{H_h^{\textup{N}}}

\newcommand{\Jeven}{J_h^{\textup{re}}}
\newcommand{\Jodd}{J_h^{\textup{ro}}}
\newcommand{\Keven}{K_h^{\textup{re}}}
\newcommand{\Kodd}{K_h^{\textup{ro}}}

\long\def\MSC#1\EndMSC{\def\arg{#1}\ifx\arg\empty\relax\else
	{\narrower\noindent%
		\small\textbf{2020 Mathematics Subject Classification:} #1} \fi}
\long\def\KEY#1\EndKEY{\def\arg{#1}\ifx\arg\empty\relax\else
	{\narrower\noindent%
		\small\textbf{Keywords:} #1\\}\fi}

\begin{document}
\allowdisplaybreaks
\title{Discrete approximations to Dirichlet and Neumann Laplacians on a half-space and norm resolvent convergence}

\author{
Horia Cornean\footnote{Department of Mathematical Sciences, Aalborg University, Skjernvej 4A, 
DK-9220 Aalborg \O{}, Denmark. Email: cornean@math.aau.dk and matarne@math.aau.dk},\;
Henrik Garde\footnote{Department of Mathematics, Aarhus University, Ny Munkegade 118, DK-8000 Aarhus C, Denmark. Email: garde@math.au.dk},\; and
Arne Jensen\footnotemark[1]
}
\date{}
\maketitle

\begin{abstract}
	\noindent We extend recent results on discrete approximations of the Laplacian in $\bR^d$ with norm resolvent convergence to the corresponding results for Dirichlet and Neumann Laplacians on a half-space. The resolvents of the discrete Dirichlet/Neumann Laplacians are embedded into the continuum using natural discretization and embedding operators. Norm resolvent convergence to their continuous counterparts is proven with a quadratic rate in the mesh size. These results generalize with a limited rate to also include operators with a real, bounded, and H\"older continuous potential, as well as certain functions of the Dirichlet/Neumann Laplacians, including any positive real power.
\end{abstract}

\KEY
norm resolvent convergence, Dirichlet Laplacian, Neumann Laplacian, lattice.
\EndKEY

\MSC
47A10, 
47A58, 
47B39
.
\EndMSC

\section{Introduction}

Let $\DL$ be the Dirichlet Laplacian and let $\NL$ be the Neumann Laplacian on the half-space $\bR^d_+$, and let $\cH^+ = L^2(\bR^d_+)$. Let $\DLh$ and $\NLh$ be the standard finite difference discretizations of $\DL$ and $\NL$, defined on $\cH_h^+ = \ell^2(h\bZ^d_+)$ with a mesh size $h>0$; see section \ref{sec:laplacians} for the precise definitions. 

Using suitable embedding operators $\Jodd, \Jeven: \cH_h^+\to\cH^+$ and discretization operators $\Kodd,\Keven: \cH^+\to \cH_h^+$ (see section~\ref{sec:embed}), we prove the following type of norm resolvent convergence with an explicit rate in the mesh size.
\begin{theorem2}
	Let $K\subset\bC\setminus[0,\infty)$ be compact. Then there exists $C>0$ such that
	\begin{equation*}
		\norm{\Jodd(\DLh-zI^+_h)^{-1}\Kodd
			-(\DL-zI^+)^{-1}}_{\cB(\cH^+)}\leq C h^2, 
	\end{equation*}
	and
	\begin{equation*}
		\norm{\Jeven(\NLh-zI^+_h)^{-1}\Keven
			-(\NL-zI^+)^{-1}}_{\cB(\cH^+)}\leq C h^2, 
	\end{equation*}
	for $0<h\leq1$ and $z\in K$.
\end{theorem2}
Norm resolvent convergence was first shown for discrete approximations of the Laplacian on $\bR^d$ in \cite{NT} and was extended to classes of Fourier multipliers in \cite{CGJ}. Recently norm resolvent convergence of discrete approximations to other operators have been considered as well, such as discrete Dirac operators in \cite{CGJ2} and quantum graph Hamiltonians in \cite{Exner}.

We prove the above result as Theorem~\ref{thm1} in section~\ref{sec:results}, and also prove several extensions to this result. In section~\ref{sec:potential} we add a real, bounded, and Hölder continuous potential $V$ to $\DL$ and $\NL$, and add a discrete potential $V_h(n) = V(hn)$ with $n\in\bZ^d_+$ to $\DLh$ and $\NLh$. The norm resolvent estimates with a potential are given in Theorem~\ref{thm:potential} with a rate that now depends explicitly on the H\"older exponent for $V$. Such norm resolvent convergence implies much improved spectral results compared to e.g.\ strong resolvent convergence. This includes convergence of the spectrum in a local Hausdorff distance \cite[Section~5]{CGJ}.

Finally in section~\ref{sec:functionalcalc} we prove norm resolvent estimates between $\Psi(\DLh)$ and $\Psi(\DL)$, and between $\Psi(\NLh)$ and $\Psi(\NL)$, defined via the functional calculus for certain functions $\Psi$ that have also been considered in \cite{CGJ} for estimates on the full space $\bR^d$. The results are given in Theorem \ref{thm2}. As an example, this includes $\Psi(\lambda) = \lambda^{s/2}$ for any positive real power $s$. This example leads to norm resolvent estimates with a rate of $h^{\min\{s,2\}}$. Fractional Laplacians on a half-space (or general domains) with Dirichlet and Neumann boundary conditions have been considered by several authors. See e.g.~\cite{A,Ghosh,GG21,GG22} for some recent results. However, results are scarce for discrete approximations of such operators.

\section{Preliminaries}\label{sect4}
We give the results in dimensions $d\geq2$. The case $d=1$ is obtained by a simple mod\-i\-fi\-ca\-tion of the arguments below.

Let $d\geq2$ and $d'=d-1$. For $x\in\bR^d$ we write $x=(x_1,x')$ with $x_1\in\bR$ and $x'\in\bR^{d'}$.
The half-space is denoted by $\bR_+^{d}=(0,\infty)\times\bR^{d'}$. 
For $x=(x_1,x')\in \bR_+^{d}$ the reflection of $x$ in the hyperplane
$\{0\}\times \bR^{d'}$ is denoted by 
\begin{equation*}
	\rfl{x}=(-x_1,x').
\end{equation*}

For $n\in\bZ^d$ we write $n=(n_1,n')$ with $n_1\in\bZ$ and $n'\in\bZ^{d'}$. We write 
\begin{equation*}
	\bZ^d_+=\{n\in\bZ^d\mid n_1\geq1\}
\end{equation*}
for the discrete half-space. We denote the reflection in the discrete hyperplane $\{0\}\times \bZ^{d'}$ by
\begin{equation*}
	\rfl{n}=(-n_1,n').
\end{equation*}

\subsection{Extension and restriction operators}
The continuous Hilbert spaces are denoted by
\begin{equation*}
\cH=L^2(\bR^d)\quad\text{and}\quad \cH^+ = L^2(\bR^d_+).
\end{equation*}
In analogy with the even-odd decomposition of functions in dimension one we introduce the reflection-even and reflection-odd functions in $\cH$ by defining
\begin{equation*}
\cHre=\{f\in\cH\mid f(x)=f(\rfl{x}),\;x\in\bR^d\}
\end{equation*}
and
\begin{equation*}
\cHro=\{f\in\cH\mid f(x)=-f(\rfl{x}),\;x\in\bR^d\},
\end{equation*}
such that $\cH=\cHre\oplus \cHro$ as an orthogonal direct sum.

The discrete Hilbert spaces are given by
\begin{equation*}
\cH_h=\ell^2(h\bZ^d)\quad\text{and}\quad \cH_h^+=\ell^2(h\bZ^d_+)
\end{equation*}
with norms
\begin{equation*}
\norm{v_h}^2_{\cH_h}=h^d\sum_{n\in\bZ^d}\abs{v_h(n)}^2
\quad\text{and}\quad
\norm{u_h}^2_{\cH_h^+}=h^d\sum_{n\in\bZ^d_+}\abs{u_h(n)}^2.
\end{equation*}
Notice that we use the index $n\in \bZ^d$ and $n\in\bZ_+^d$ in the notation for $v_h\in \cH_h$ and $u_h\in \cH_h^+$. The dependence on the mesh size is given by the subscript $h$.

The reflection-even and reflection-odd sequences are defined by
\begin{equation*}
\cHreh=\{v_h\in\cH_h\mid v_h(n)=v_h(\rfl{n}),\;n\in\bZ^d\}
\end{equation*}
and
\begin{equation*}
\cHro_h=\{v_h\in\cH_h\mid v_h(n)=-v_h(\rfl{n}),\;n\in\bZ^d\}.
\end{equation*}
We have $\cH_h=\cHre_h\oplus \cHro_h$ as an orthogonal direct sum.

The reflection-odd extension operator $\odd\colon\cH^+\to\cH$ and reflection-even extension operator $\even\colon\cH^+ \to \cH$ are given by
\begin{equation*}
	\odd f(x)=\begin{dcases}
f(x), & x\in\bR^d_+,\\
-f(\rfl{x}), & \rfl{x}\in\bR^d_+,
\end{dcases} 
	\qquad 
	\even f(x)=\begin{dcases}
f(x), & x\in\bR^d_+,\\
f(\rfl{x}), & \rfl{x}\in\bR^d_+.
\end{dcases}
\end{equation*}
In the discrete case the reflection-odd extension operator $\odd_h\colon \cH^+_h\to\cH_h$ is given by
\begin{equation*}
	\odd_h u_h(n)=\begin{dcases}
u_h(n), & n\in\bZ^d_+,\\
0, & \text{$n_1=0$, $n'\in\bZ^{d'}$,}\\
-u_h(\rfl{n}), & \rfl{n}\in\bZ^d_+.
\end{dcases}
\end{equation*}
The discrete reflection-even extension operator 
$\even_h\colon\cH_h^+\to\cH_h$ is defined by
\begin{equation*}
	\even_h u_h(n)=\begin{dcases}
u_h(n), & n\in\bZ^d_+,\\
u_h(1,n'), & \text{$n_1=0$, $n'\in\bZ^{d'}$,}\\
u_h(\rfl{n}), & \rfl{n}\in\bZ^d_+.
\end{dcases}
\end{equation*}

The natural restriction operators onto the half-spaces are denoted by
\begin{equation*}
	\R\colon \cH\to\cH^+\quad\text{and}\quad
\R_h\colon \cH_h\to\cH^+_h.
\end{equation*}
Obviously we have $\R\odd = \R\even = I^+$ and $\R_h\odd_h = \R_h\even_h=I^+_h$, where we also introduced the notation for the identity operators on $\cH^+$ and $\cH^+_h$, respectively.

\subsection{Embedding and discretization operators} \label{sec:embed}

In~\cite{CGJ} embedding and discretization operators were defined using a pair of biorthogonal Riesz sequences. Here we consider only the special case of an orthogonal sequence, as in~\cite{NT}, but with the additional assumption that the generating function is reflection-even.

\begin{assumption}\label{assumpd}
	Assume $\vp_0\in \cHre$ such that $\{\vp_0(\,\cdot\,-n)\}_{n\in\bZ^d}$ is an orthonormal sequence in $\cH$.
\end{assumption}
Define 
\begin{equation*}
	\vp_{h,n}(x)=\vp_0((x-hn)/h), \quad h>0, \: n\in\bZ^d,\; x\in\bR^d.
\end{equation*}
Since $\vp_0$ is assumed reflection-even we have the 
important property
\begin{equation}\label{symd}
	\vp_{h,n}(\rfl{x})=\vp_{h,\rfl{n}}(x), \quad h>0, \: 
n\in\bZ^d,\; x\in\bR^d.
\end{equation}
Define the embedding operators $J_h\colon\cH_h\to\cH$ by
\begin{equation*}
	J_hv_h(x)=\sum_{n\in\bZ^d}v_h(n)\vp_{h,n}(x),\quad v_h\in \cH_h.
\end{equation*}
From Assumption~\ref{assumpd} it follows that $\{h^{-d/2}\vp_{h,n}\}_{n\in\bZ^d}$
is an orthonormal sequence, hence that $J_h$ is isometric.

The discretization operators are given by $K_h=(J_h)^{\ast}$. With the convention that inner products are linear in the second entrance, we explicitly have
\begin{equation*}
	K_hg(n)=\frac{1}{h^d}\ip{\vp_{h,n}}{g}_{\cH},\quad g\in\cH.
\end{equation*}
Let us note that~\eqref{symd} implies $J_h\cHroh\subseteq\cHro$, $J_h\cHreh\subseteq \cHre$, $K_h\cHro\subseteq\cHroh$, and $K_h\cHre\subseteq\cHreh$.

The half-space embedding operators $\Jodd, \Jeven\colon\cH^+_h\to\cH^+$ are defined as
\begin{equation*}
	\Jodd = \R J_h \odd_h, \qquad \Jeven = \R J_h \even_h.
\end{equation*}
The operators $\Jodd$ and $\Jeven$ are isometric, as can be seen from the following computation. Let $u_h\in\cH_h^+$ and use that $J_h\odd_h u_h\in\cHro$,
\begin{equation*}
	\norm{\Jodd u_h}^2_{\cH^+}=\tfrac{1}{2}\norm{J_h\odd_h u_h}^2_{\cH}
=\tfrac{1}{2}\norm{\odd_h u_h}^2_{\cH_h}=
\norm{u_h}^2_{\cH_h^+}.
\end{equation*}
A similar computation holds for $\Jeven$.

The half-space discretization operators
$\Kodd,\Keven\colon\cH^+\to\cH_h^+$
are defined as 
\begin{equation*}
	\Kodd = \R_h K_h\odd, \qquad \Keven = \R_h K_h\even.
\end{equation*}
Note that $\Kodd\Jodd = \Keven\Jeven=I^+_h$. $\Jodd \Kodd$ is the orthogonal projection onto $\ran \Jodd$ in $\cH^+$ and $\Jeven\Keven$ is the orthogonal projection onto $\ran \Jeven$ in $\cH^+$.

\subsection{Laplacians} \label{sec:laplacians}

Let $H_0=-\Delta$ be the Laplacian in $\cH$ with domain $\cD(H_0)=H^2(\bR^d)$. 

The Dirichlet Laplacian $\DL$ on $\cH^+$ is defined as the positive self-adjoint operator given by the Friedrichs extension of $-\Delta|_{C_0^\infty(\bR_+^d)}$. Equivalently, $\DL$ is the variational operator associated with the triple $(\cH^+,H_0^1(\bR_+^d), q)$, where the sesquilinear form $q$ is
\begin{equation*}
	q(u,v) = \int_{\bR_+^d} \overline{\nabla u}\cdot \nabla v\,\di x.
\end{equation*} 
By \cite[Theorem~9.11]{Grubb} the domain of $\DL$ on a half-space simplifies to
\begin{equation*}
	\cD(\DL) = H^2(\bR^d_+)\cap H_0^1(\bR^d_+) = \{u\in H^2(\bR^d_+) \mid \gamma_0 u = 0\},
\end{equation*}
where $\gamma_0$ is the Dirichlet trace operator.

Next we define the Neumann Laplacian $\NL$ on $\cH^+$ as the positive self-adjoint variational operator associated with the triple $(\cH^+, H^1(\bR^d_+),q)$. On a half-space its domain simplifies via \cite[Theorem~9.20]{Grubb} to
\begin{equation*}
	\cD(\NL) = \{u\in H^2(\bR^d_+) \mid \gamma_1 u = 0\},
\end{equation*}
where $\gamma_1$ is the Neumann trace operator.

From \cite[Theorem~9.2]{Grubb} the trace maps satisfy $\gamma_j\in\cB(H^m(\bR^d_+),H^{m-j-\frac{1}{2}}(\partial\bR^d_+))$ for $m\in \bN$ and $j\leq m-1$.

We need the following lemma. The result is a consequence of e.g.~\cite[Proposition 2.2]{MS}. We give a shorter proof for the sake of completeness.

\begin{lemma}\label{lemma22} {}\
	\begin{enumerate}[\rm(i)]
		\item Let $f\in \cD(\DL)$. Then $\odd f\in\cD(H_0)$ and $\odd \DL f = H_0 \odd f$. Furthermore, for $z\in\bC\setminus[0,\infty)$ and all $g\in\cH^+$ we have 
		\begin{equation}\label{interd}
			\odd (\DL-zI^+)^{-1}g=(H_0-zI)^{-1}\odd g.
		\end{equation}
		\item Let $f\in \cD(\NL)$. Then $\even f\in\cD(H_0)$ and $\even \NL f = H_0 \even f$. Furthermore, for $z\in\bC\setminus[0,\infty)$ and all $g\in\cH^+$ we have 
		\begin{equation}\label{intern}
			\even (\NL-zI^+)^{-1}g=(H_0-zI)^{-1}\even g.
		\end{equation}
	\end{enumerate}
\end{lemma}
\begin{proof}
	(i): Let $f\in\cD(\DL)$.
	Since $C_0^{\infty}(\bR_+^d)$ is a core for $\DL$, we can 
	find a sequence $\psi_n\in C_0^{\infty}(\bR_+^d)$ such that
	$\psi_n\to f$ and $\DL\psi_n\to \DL f$ in $\cH^+$, as $n\to\infty$. 
	We have $\odd\psi_n\in C_0^{\infty}(\bR^d)$, such that $\odd\psi_n\to\odd f$ and $\odd \DL\psi_n\to \odd \DL f$ in $\cH$, as $n\to\infty$. Note that
	\begin{equation*}
		\odd \DL\psi_n=\odd(-\Delta\psi_n)=-\Delta\odd\psi_n,
	\end{equation*}
	since $-\Delta$ commutes with orthogonal coordinate transformations and since $\psi_n$ is supported in $\bR^d_+$, i.e.\ away from the hyperplane $\{0\}\times\bR^{d'}$. Thus
	\begin{equation*}
		\odd\psi_n\to \odd f \quad \text{and} \quad -\Delta\odd\psi_n\to \odd \DL f
	\end{equation*}
	in $\cH$. Since $H_0$ is a closed operator we conclude that
	$\odd f\in\cD(H_0)$ and $H_0\odd f=\odd \DL f$. The second part of the statement follows by using $\cD(\DL)=\ran((\DL-zI^+)^{-1})$ for $z\in\bC\setminus[0,\infty)$.
	
	(ii): Let $f\in \cD(\NL)$. Restrictions of $\even f$ to either side of $\{0\}\times \bR^{d'}$ has coinciding Dirichlet and Neumann traces, so at least $\even f \in H^1(\bR^d)$. We can approximate $f$ in $H^1(\bR^d_+)$ by a sequence $\psi_n\in C^\infty(\overline{\bR^d_+})$ with $\gamma_1\psi_n = 0$. Now $\even \psi_n \in C^1(\bR^d)$ implies the identity 
	\begin{equation*}
		\partial_1 \even f = \odd \partial_1 f.
	\end{equation*}
	However we have that $\partial_1 f\in H_0^1(\bR^d_+)$, which has a zero-extension $E_0(\partial_1 f)\in H^1(\bR^d)$. Since
	\begin{equation*}
		\partial_1\even f(x) = \odd\partial_1 f(x) = E_0(\partial_1 f)(x) - E_0(\partial_1 f)(\rfl{x}),
	\end{equation*}
	then $\partial_1\even f \in H^1(\bR^d)$ and as a consequence $\even f\in H^2(\bR^d) = \cD(H_0)$, since there was no contention regarding the square integrability of the other partial derivatives. The rest of the proof follows by using that $-\Delta$ on $H^2(\bR^d)$ commutes with orthogonal coordinate transformations, and that $\cD(\NL)=\ran((\NL-zI^+)^{-1})$ for $z\in\bC\setminus[0,\infty)$.
\end{proof}


The discrete Laplacian $H_{0,h}$ on $\cH_h$ is given by
\begin{equation*}
H_{0,h}v_h(n)=\frac{1}{h^2}
\sum_{j=1}^d\bigl(2v_h(n)-v_h(n+e_j)-v_h(n-e_j)\bigr),
\quad v_h\in\cH_h,\;n\in\bZ^d.
\end{equation*}
Here $\{e_j\}_{j=1}^d$ denotes the canonical basis for $\bR^d$.
The discrete Dirichlet Laplacian on $\cH_h^+$ is given by
\begin{equation*}
\DLh v_h(n)=
\begin{dcases}
\frac{1}{h^2}\sum_{j=2}^{d}\bigl(2v_h(n)-v_h(n+e_j)-v_h(n-e_j)\bigr)
& \\
\quad+
\frac{1}{h^2}\bigl(2v_h(n)-v_h(n+e_1)\bigr), & n_1=1,\\
\frac{1}{h^2}
\sum_{j=1}^d\bigl(2v_h(n)-v_h(n+e_j)-v_h(n-e_j)\bigr), & n_1\geq2.
\end{dcases}
\end{equation*}
Let $u_h\in\cH_h^+$. Then using the definitions
one can verify that $\odd_h \DLh u_h=H_{0,h}\odd_h u_h$ and then
\begin{equation}\label{interh}
\odd_h (\DLh-zI^+_h)^{-1}u_h=(H_{0,h}-zI_h)^{-1}\odd_h u_h.
\end{equation}
The discrete Neumann Laplacian on $\cH_h^+$ is given by
\begin{equation*}
	\NLh v_h(n)=
	\begin{dcases}
		\frac{1}{h^2}\sum_{j=2}^{d}\bigl(2v_h(n)-v_h(n+e_j)-v_h(n-e_j)\bigr)
		& \\
		\quad+
		\frac{1}{h^2}\bigl(v_h(n)-v_h(n+e_1)\bigr), & n_1=1,\\
		\frac{1}{h^2}
		\sum_{j=1}^d\bigl(2v_h(n)-v_h(n+e_j)-v_h(n-e_j)\bigr), & n_1\geq2.
	\end{dcases} 
\end{equation*}
Let $u_h\in\cH_h^+$. Similar to the above, $\even_h \NLh u_h=H_{0,h}\even_h u_h$ and then
\begin{equation}\label{interh2}
	\even_h (\NLh-zI^+_h)^{-1}u_h=(H_{0,h}-zI_h)^{-1}\even_h u_h.
\end{equation}

\begin{remark}
	Since we use homogeneous Dirichlet and Neumann conditions, the discrete Laplacians have a very similar finite difference structure. The discrete Neumann Laplacian only differs from the discrete Dirichlet Laplacian at the indices where $n_1 = 1$. Here the contributions from the boundary conditions either mean that $v_h(n-e_1)=0$ (Dirichlet case) or that $v_h(n-e_1) = v_h(n)$ (Neumann case). This subtle difference also implies the connections to odd and even reflections in \eqref{interh} and \eqref{interh2}.
\end{remark}

\section{Results} \label{sec:results} 

Additional assumptions on the function $\vp_0$ are needed to obtain our results, cf.~\cite[Assumption~2.8]{CGJ} 
or~\cite[Assumption~B]{NT}. Let $\widehat{\vp}_0$ denote the Fourier transform of $\vp_0$, defined as
\begin{equation*}
	\widehat{\vp}_0(\xi)=(2\pi)^{-d/2}\int_{\bR^d}\e^{-\I x\cdot \xi}\vp_0(x)\di x.
\end{equation*}
\begin{assumption}\label{assump2.8}
Let $\vp_0$ satisfy Assumption~\ref{assumpd} and assume that
 $\widehat{\vp}_0$ is essentially bounded.
 Assume there exists $c_0>0$ such that
\begin{equation*}
\supp(\widehat{\vp}_0)\subseteq[-3\pi/2,3\pi/2]^d,
\quad\text{and}\quad
\abs{\widehat{\vp}_0(\xi)}\geq c_0,\quad
\xi\in[-\pi/2,\pi/2]^d.
\end{equation*}
\end{assumption}

\begin{theorem}\label{thm1}
Let $\Jodd$, $\Jeven$, $\Kodd$, and $\Keven$ be as above, with $\vp_0$ satisfying
 Assumption~\ref{assump2.8}.
Let $K\subset\bC\setminus[0,\infty)$ be compact. Then there exists $C>0$ such that
\begin{equation}
\norm{\Jodd(\DLh-zI^+_h)^{-1}\Kodd
-(\DL-zI^+)^{-1}}_{\cB(\cH^+)}\leq C h^2, \label{maind}
\end{equation}
and
\begin{equation}
\norm{\Jeven(\NLh-zI^+_h)^{-1}\Keven
	-(\NL-zI^+)^{-1}}_{\cB(\cH^+)}\leq C h^2, \label{mainn}
\end{equation}
for $0<h\leq1$ and $z\in K$.
\end{theorem}
\begin{proof}
Let $f\in\cH^+$. Then 
\begin{multline*}
\norm{\Jodd(\DLh-zI^+_h)^{-1}\Kodd f
-(\DL-zI^+)^{-1}f}_{\cH^+}^2\\
=\tfrac{1}{2}\norm{\odd\Jodd(\DLh-zI^+_h)^{-1}\Kodd f
-\odd(\DL-zI^+)^{-1}f}_{\cH}^2.
\end{multline*}
We have $\odd \Jodd=\odd \R J_h\odd_h = J_h\odd_h$, since $J_h\odd_h u_h$ is a reflection-odd function for all $u_h\in \cH_h^+$.
Thus using \eqref{interh} we get
\begin{align*}
\odd\Jodd(\DLh-zI^+_h)^{-1}\Kodd f &=
J_h\odd_h (\DLh-zI^+_h)^{-1}\Kodd f\\
&=J_h(H_{0,h}-zI_h)^{-1}\odd_h\Kodd f.
\end{align*}
Now $\odd_h\Kodd f=\odd_h\R_h K_h \odd f=K_h\odd f$, since $K_h\odd f$ is a reflection-odd sequence. Thus we have shown
\begin{equation*}
\odd\Jodd(\DLh-zI^+_h)^{-1}\Kodd f 
=J_h(H_{0,h}-zI_h)^{-1}K_h\odd f.
\end{equation*}
Using this result together with \eqref{interd} we have shown that
\begin{multline*}
\normb{\bigl(\Jodd(\DLh-zI^+_h)^{-1}\Kodd 
-(\DL-zI^+)^{-1}\bigr)f}_{\cH^+}^2\\
=\tfrac{1}{2}\normb{\bigl(J_h(H_{0,h}-zI_h)^{-1}K_h
-(H_0-zI)^{-1}\bigr)\odd f}_{\cH}^2.
\end{multline*}
Thus we can use the results in~\cite{CGJ} or~\cite{NT} to obtain \eqref{maind}.

To prove \eqref{mainn} note that $\even\Jeven = J_h\even_h$ and $\even_h\Keven = K_h\even$, and then use
 \eqref{intern} instead of \eqref{interd} and \eqref{interh2} instead of \eqref{interh}. This leads to:
\begin{multline*}
	\normb{\bigl(\Jeven(\NLh-zI^+_h)^{-1}\Keven 
		-(\NL-zI^+)^{-1}\bigr)f}_{\cH^+}^2\\
	=\tfrac{1}{2}\normb{\bigl(J_h(H_{0,h}-zI_h)^{-1}K_h
		-(H_0-zI)^{-1}\bigr)\even f}_{\cH}^2,
\end{multline*}
which together with the results in~\cite{CGJ} or~\cite{NT} completes the proof of \eqref{mainn}.
\end{proof}

\subsection{Adding a potential} \label{sec:potential}

Next we add a potential. To obtain the results we introduce two assumptions.
\begin{assumption}\label{assumpdecay}
Let $\tau>d$. Assume that there exists $C>0$ such that
\begin{equation*}
\abs{\vp_0(x)}\leq C (1+\abs{x})^{-\tau},\quad x\in\bR^d.
\end{equation*}
\end{assumption}
\begin{assumption}\label{assumpV}
Let $V\colon \overline{\bR^d_+}\to\bR$ be a bounded function which is uniformly H\"{o}lder continuous of order $\theta\in(0,1]$.
\end{assumption}
Note that $\overline{\bR^d_+}$ denotes the closed half-space, so the conditions hold up to the boundary.

\begin{lemma}
Let $V$ satisfy Assumption~\ref{assumpV}. Then $\even V$ is bounded and uniformly H\"{o}lder continuous of order $\theta$ on $\bR^d$.
\end{lemma}
\begin{proof}
Boundedness is clear, and for the H\"older continuity we only need to consider points $x,y\in\bR^d$ such that $x,\rfl{y}\in\bR^d_+$. Now Assumption~\ref{assumpV} and $\abs{x-\rfl{y}}\leq \abs{x-y}$ imply
\begin{equation*}
\abs{\even V(x)-\even V(y)} = \abs{V(x)-V(\rfl{y})} \leq C\abs{x-\rfl{y}}^\theta \leq C\abs{x-y}^\theta. \qedhere
\end{equation*} 
\end{proof}

We define the discretized potential as $V_h(n)=V(hn)$, $h>0$, $n\in\bZ_+^d$. Then we define $\DLp=\DL+V$ and $\NLp = \NL+V$ on $\cH^+$, and 
$\DLhp=\DLh+V_h$ and $\NLhp=\NLh+V_h$ on $\cH_h^+$.

\begin{theorem} \label{thm:potential}
Let $\Jodd$, $\Jeven$, $\Kodd$, and $\Keven$ be as above, with $\vp_0$ satisfying
 Assumptions~\ref{assump2.8} and~\ref{assumpdecay}. Let $V$ satisfy Assumption~\ref{assumpV}. Define
\begin{equation*}
\frac{1}{\theta'}=\frac{1}{\theta}+\frac{1}{\tau-d}.
\end{equation*}
Let $K\subset\bC\setminus\bR$ be compact. Then there exists $C>0$ such that
\begin{equation}\label{maindV}
\norm{\Jodd(\DLhp-zI^+_h)^{-1}\Kodd
-(\DLp-zI^+)^{-1}}_{\cB(\cH^+)}\leq C h^{\theta'},
\end{equation}
and
\begin{equation}\label{mainnV}
\norm{\Jeven(\NLhp-zI^+_h)^{-1}\Keven
	-(\NLp-zI^+)^{-1}}_{\cB(\cH^+)}\leq C h^{\theta'},
\end{equation}
for $0<h\leq1$ and $z\in K$.
\end{theorem}
\begin{proof}
Let $f\in\cH^+$ and $x\in\bR^d_+$. Then 
\begin{equation*}
\odd (Vf)(x)=V(x)f(x)=(\even V)(x)(\odd f)(x),
\end{equation*}
and
\begin{equation*}
\odd (Vf)(\rfl{x})=-(Vf)(x)=V(x)(-f(x))
=(\even V)(\rfl{x})(\odd f)(\rfl{x}).
\end{equation*}
Thus we have $\odd V=(\even V)\odd$ as operators from $\cH^+$ to $\cH$, 
where $\even V$ denotes the operator of multiplication in $\cH$ by
$(\even V)(x)$, $x\in\bR^d$.
Let $H=H_0+\even V$ on
$\cH$. Then combining the above result with
the arguments leading to~\eqref{interd} we get for $f\in\cH^+$
\begin{equation*}
\odd (\DLp-zI^+)^{-1}f=(H-zI)^{-1}\odd f.
\end{equation*}
We can repeat these arguments in the discrete case, leading to
\begin{equation*}
\odd_h (\DLhp-zI^+_h)^{-1}u_h=(H_h-zI_h)^{-1}\odd_h u_h
\end{equation*}
for $u_h\in \cH_h^+$. Here we have defined $H_h=H_{0,h}+\even_h V_h$ on
$\cH_h$. Note that $\even_h V_h$ and $(\even V)_h$ may differ only at $n_1=0$. Thus replacing $\even_h V_h$ by $(\even V)_h$ introduces an error of order $h^{\theta}$, due to Assumption~\ref{assumpV}, and this error can be absorbed in the final estimate below.

Repeating the computations in the proof of Theorem~\ref{thm1} we get
for $f\in\cH^+$
\begin{multline*}
\normb{\bigl(\Jodd(\DLhp-zI^+_h)^{-1}\Kodd 
-(\DLp-zI^+)^{-1}\bigr)f}_{\cH^+}^2\\
=\tfrac{1}{2}\normb{\bigl(J_h(H_{h}-zI_h)^{-1}K_h
-(H-zI)^{-1}\bigr)\odd f}_{\cH}^2.
\end{multline*}
We can then use~\cite[Theorem~4.4]{CGJ} to complete the proof. 

The proof for the Neumann Laplacian is analogous, using instead that $\even (Vf) = (\even V)(\even f)$, which for $f\in\cH^+$ and $u_h\in \cH^+_h$ gives
\begin{align*}
	\even (\NLp-zI^+)^{-1}f &= (H-zI)^{-1}\even f, \\
	\even_h (\NLhp-zI^+_h)^{-1}u_h &= (H_h-zI_h)^{-1}\even_h u_h.
\end{align*}
This leads to
\begin{multline*}
	\normb{\bigl(\Jeven(\NLhp-zI^+_h)^{-1}\Keven 
		-(\NLp-zI^+)^{-1}\bigr)f}_{\cH^+}^2\\
	=\tfrac{1}{2}\normb{\bigl(J_h(H_{h}-zI_h)^{-1}K_h
		-(H-zI)^{-1}\bigr)\even f}_{\cH}^2,
\end{multline*}
which can also be estimated by~\cite[Theorem~4.4]{CGJ}.
\end{proof}

\subsection{Functions of Dirichlet and Neumann Laplacians} \label{sec:functionalcalc}

Now we extend the approximation results given in Theorem~\ref{thm1} to functions of the Dirichlet and Neumann Laplacians on the half-space. Let $\Psi\colon [0,\infty)\to\bR$ be a Borel function. Using the functional calculus we can define the operators $\Psi(\DL)$, $\Psi(\DLh)$, $\Psi(\NL)$, and $\Psi(\NLh)$. 

We need the following lemma, which is an immediate consequence of~\cite[Proposition~5.15]{Sch}; see also \cite{MS}. For operators $S$ and $T$ the notation $S\subset T$ means that $T$ is an extension of $S$.
\begin{lemma}\label{lemma41}
For $j=1,2$ assume that $A_j$ is a self-adjoint operator on a Hilbert space $\cH_j$. Assume that $B\colon \cH_1\to\cH_2$ is a bounded operator such that
\begin{equation*}
BA_1\subset A_2B.
\end{equation*}
Let $\Psi$ be a Borel function on $\bR$. Then we have
\begin{equation}\label{eq42}
B\Psi(A_1)\subset \Psi(A_2)B.
\end{equation}
If $\Psi$ is a bounded function then equality holds in~\eqref{eq42}.
\end{lemma}

In the following assumption the parameters are chosen to be compatible with the ones in \cite[Assumption~3.1]{CGJ}.
 
\begin{assumption}\label{assumpPsi}
Assume
\begin{equation*}
\alpha>\tfrac12,\quad \beta> -\tfrac12,\quad\text{and}\quad
\alpha\leq1+\beta<2\alpha\leq3+\beta.
\end{equation*}
Let $\Psi\colon [0,\infty)\to\bR$ be a continuous function which is continuously differentiable on $(0,\infty)$ and satisfies the following conditions:
\begin{enumerate}
\item[(1)] $\Psi(0)=0$,
\item[(2)] there exist $c_0>0$ and $c_1>0$ such that 
$\Psi(\lambda)\geq c_0\lambda^{\alpha/2}$ for $\lambda\geq c_1$.
\item[(3)] there exists $c>0$ such that $\abs{\Psi'(\lambda)}\leq c
\lambda^{(\beta-1)/2}$ for $\lambda>0$.
\end{enumerate}
\end{assumption}

We omit the straightforward proof of the following lemma.
\begin{lemma}\label{lemmaPsi}
Let $\Psi$ satisfy Assumption~\ref{assumpPsi} with parameters $\alpha$ and $\beta$. Define $G_0(\xi)=\Psi(\abs{\xi}^2)$, $\xi\in\bR^d$. Then $G_0$ satisfies Assumption~\textup{3.1} in~\textup{\cite{CGJ}} with the same parameters $\alpha$ and $\beta$.
\end{lemma}

Next we define
\begin{equation*}
G_{0,h}(\xi)=G_0(\tfrac{2}{h}\sin(\tfrac{h}{2}\xi_1),
\tfrac{2}{h}\sin(\tfrac{h}{2}\xi_2),\ldots,
\tfrac{2}{h}\sin(\tfrac{h}{2}\xi_d)),\quad h>0,\quad \xi\in\bR^d.
\end{equation*}
Using these definitions it follows that $\Psi(H_0)$ is the Fourier multiplier with symbol $G_0$ on $\cH$, and $\Psi(H_{0,h})$ is the Fourier multiplier with symbol $G_{0,h}$ on $\cH_h$. The operators $\Psi(\DL)$ and $\Psi(\NL)$ on $\cH^+$, and the operators
$\Psi(\DLh)$ and $\Psi(\NLh)$ on $\cH^+_h$, are defined using the functional calculus.

We have the following extension of Theorem~\ref{thm1}.
\begin{theorem}\label{thm2}
Let $\Psi$ satisfy Assumption~\ref{assumpPsi} with parameters $\alpha$ and $\beta$. Let
\begin{equation*}
	\gamma=\min\{2\alpha-1,2\alpha-\beta-1\}.
\end{equation*}
Let $\Jodd$, $\Jeven$, $\Kodd$, and $\Keven$ be as above, with $\vp_0$ satisfying Assumption~\ref{assump2.8}. Let $K\subset\bC\setminus[0,\infty)$ be compact. Then there exists $C>0$ such that
\begin{equation}
	\norm{\Jodd(\Psi(\DLh)-zI^+_h)^{-1}\Kodd
-(\Psi(\DL)-zI^+)^{-1}}_{\cB(\cH^+)}\leq C h^{\gamma}, \label{mainda}
\end{equation}
and
\begin{equation}
	\norm{\Jeven(\Psi(\NLh)-zI^+_h)^{-1}\Keven
	-(\Psi(\NL)-zI^+)^{-1}}_{\cB(\cH^+)}\leq C h^{\gamma}, \label{mainna}
\end{equation}
for $0<h\leq1$ and $z\in K$.
\end{theorem}
\begin{proof}
We prove the result for the Dirichlet Laplacians.
Assumption~\ref{assumpPsi} and Lemma~\ref{lemmaPsi} together with \cite[Proposition~3.5]{CGJ} imply that we have the estimate
\begin{equation}\label{CGJ-est}
\norm{
J_h(\Psi(H_{0,h})-zI_h)^{-1}K_h-(\Psi(H_0)-zI)^{-1}
}_{\cB(\cH)}\leq C h^{\gamma}
\end{equation}
for $0<h\leq1$ and $z\in K$, with $K$ satisfying the assumption in the theorem. 

Combine Lemma~\ref{lemma22} with Lemma~\ref{lemma41} to get the result
\begin{equation}\label{interda}
\cO(\Psi(\DL)-zI^+)^{-1}=(\Psi(H_0)-zI)^{-1}\cO,\quad z\in K.
\end{equation}
Analogously, using \eqref{interh} and Lemma~\ref{lemma41} we get
\begin{equation}\label{interdah}
\cO_h(\Psi(\DLh)-zI^+_h)^{-1}=(\Psi(H_{0,h})-zI_h)^{-1}\cO_h,
\quad z\in K.
\end{equation}
Using the results \eqref{CGJ-est}--\eqref{interdah} we can repeat the arguments in the proof of Theorem~\ref{thm1} to get the result in the Dirichlet case. The proof in the Neumann case is almost the same, so we omit it.
\end{proof}

\begin{remark}
	By repeating the proof of Theorem~\ref{thm:potential}, we may also add a potential $V$ to the operators $\Psi(\DL)$ and $\Psi(\NL)$ and add a discrete potential $V_h$ to the operators $\Psi(\DLh)$ and $\Psi(\NLh)$. The resulting estimates, replacing those in \eqref{mainda} and \eqref{mainna}, will have the rate $h^{\min\{\gamma,\theta'\}}$.
\end{remark}

Of particular interest are the functions $\Psi$ that give the powers of the Laplacian $H_0$.
Let $s>0$ and define $\Psi_s(\lambda)=\lambda^{s/2}$, $\lambda\geq0$. Then $G_0(\xi)=\abs{\xi}^s$ and $\Psi_s(H_0)=(-\Delta)^{s/2}$. For $s\geq2$ we can take $\alpha=(s+2)/2$ and $\beta=s-1$ to satisfy the conditions in Assumption~\ref{assumpPsi}. Then the estimate~\eqref{CGJ-est} holds with $\gamma=2$.

For $\frac12<s<2$ the conditions in Assumption~\ref{assumpPsi} are satisfied with $\alpha=s$ and $\beta=s-1$. We get $\gamma=s$ for $1\leq s<2$. For $0<s<1$ we can use the result in \cite[Proposition~3.11]{CGJ} which yields the estimate~\eqref{CGJ-est} for $\Psi_s(H_0)$ and $\Psi_s(H_{0,h})$ with $\gamma=s$.

We summarize the results above as a Corollary to both Theorem~\ref{thm2} and the results in~\cite{CGJ}.

\begin{corollary}
Let $\Psi_s(\lambda)=\lambda^{s/2}$, $\lambda\geq0$, $s>0$. Then the estimates~\eqref{mainda} and~\eqref{mainna} hold for $\gamma=\min\{s,2\}$.
\end{corollary}

\begin{remark}\label{rem}
The operators $\Psi_s(\DL)$ defined here do not agree with the fractional Dirichlet Laplacians on a half-space defined in~\cite{GG21,GG22}.
Let $u\in\cD(\Psi_s(\DL))$, then Lemmas~\ref{lemma22} and~\ref{lemma41} imply $\odd\Psi_s(\DL) u= \Psi_s(H_0)\odd u$, such that
$\Psi_s(\DL) u=\R \Psi_s(H_0)\odd u$. Whereas in~\cite{GG21,GG22} the definition is based on the operator $\R \Psi_s(H_0)E_0$ applied to suitable functions in $\cH^+$, where $E_0$ is the operator for extension by zero. Hence the two approaches differ by the type of extension operator that is used.
\end{remark} 

\paragraph*{Acknowledgments.} This research is partially supported by grant 8021--00084B from Independent Research Fund Denmark \textbar\ Natural Sciences.

\end{document}